\documentclass[11pt]{article}
\usepackage{style}
\usepackage{graphicx}
\usepackage{amsthm, amsmath, amssymb, tikz, bm}
\usepackage{mathtools, intcalc}
\usepackage{xifthen}
\usepackage[colorlinks=true,
linkcolor=blue,citecolor=blue,
urlcolor=blue]{hyperref}

\usepackage[shortlabels]{enumitem}
\usetikzlibrary{calc,shapes, backgrounds}
\allowdisplaybreaks

\newcommand{\ds}{\displaystyle}
\newcommand{\dss}{\displaystyle\sum}

\newcommand{\lp}{\left(}
\newcommand{\rp}{\right)}

\title{Outerplanar graphs with positive Lin--Lu--Yau curvature}
\author{George Brooks
	\thanks{University of South Carolina, Columbia, SC, USA ({\tt ghbrooks@email.sc.edu})}
	\quad
	Fadekemi Osaye
	\thanks{Troy University, Troy, AL, USA ({\tt fosaye@troy.edu}) }
	\quad
	Anna Schenfisch
	\thanks{Eindhoven University of Technology, Eindhoven, Netherlands ({\tt a.k.schenfisch@tue.nl})}
	\quad
	Zhiyu Wang
	\thanks{Louisiana State University, Baton Rouge, LA, USA ({\tt zhiyuw@lsu.edu}) }
	\quad
	Jing Yu
	\thanks{Shanghai Center for Mathematical Sciences, Fudan University, Shanghai, China ({\tt jyu@fudan.edu.cn}) }
}
\begin{document}
	
\maketitle

 \abstract{In this paper, we show that all simple outerplanar graphs $G$ with minimum degree at least $2$ and positive Lin--Lu--Yau Ricci curvature on every edge have maximum degree at most $9$. Furthermore, if~$G$ is maximally outerplanar, then $G$ has at most $10$ vertices. Both upper bounds are sharp.}
 
\section{Introduction}

Ricci curvature plays a crucial role in the geometric analysis of Riemannian manifolds. Roughly speaking, it measures the degree to which the geometry of a metric tensor differs locally from that of a Euclidean space.  Curvature in the continuous setting has been studied in depth throughout the past 200 years, and more recent interest has arisen in establishing analogous results in metric spaces. The definition of the (non-combinatorial) Ricci curvature on metric spaces first came from the Bakry and \'Emery notation \cite{Bakry-Emery} who defined the ``lower Ricci curvature bound" through the heat semigroup~$(P_t)_{t\geq 0}$ on a metric measure space. Ollivier \cite{Ollivier} defined the coarse Ricci curvature of metric spaces in terms of how much small balls are closer (in Wasserstein transportation distance) than their centers are. This notion of coarse Ricci curvature on discrete spaces was also made explicit in the Ph.D. thesis of Sammer~\cite{Sammer}.
The first definition of Ricci curvature on graphs was introduced by Chung and Yau in~\cite{Chung-Yau96}. To obtain a good log-Sobolev inequality, they defined \textit{Ricci-flatness} in graphs. Later, Lin and Yau \cite{LY} gave a generalization of the lower Ricci curvature bound in the framework of graphs. In \cite{LLY}, Lin, Lu, and Yau modified Ollivier's Ricci curvature \cite{Ollivier} and defined a new variant of Ricci curvature on graphs, which does not depend on the idleness of the random walk. For other variants of curvature on discrete spaces, see e.g.,~\cite{Devriendt-Lambiotte2022, Forman2003} and the references therein.

In this paper, we are interested in studying \textit{planar} and \textit{outerplanar} graphs with positive curvature. A planar graph is a graph that can be embedded in the plane, i.e., it can be drawn on the plane in such a way that its edges intersect only at their endpoints. Infinite planar graphs are often treated as the discrete version of noncompact simply connected $2$-dimensional manifolds. Thus it's natural to consider the `curvature' of planar graphs. A notion of curvature, previously more extensively studied in graph theory, is the \textit{combinatorial curvature}, which is defined on the vertices of a graph.
Given a graph $G$ that is $2$-cell embedded in a surface without loops or multiple edges, the \textit{combinatorial curvature} of $v\in V(G)$ is defined as
$K(v):= 1-\frac{d_v}{2} + \dss_{\sigma\in F(v)} \frac{1}{|\sigma|}$,
where $d_v$ denotes the degree of $v$, $F(v)$ is the multiset of faces touching $v$ and $|\sigma|$ is the size of the face $\sigma$.  A graph $G$ is said to have positive combinatorial curvature everywhere if $\phi(v) > 0$ for all $v\in V(G)$.
Higuchi \cite{Higuchi01} conjectured that if $G$ is a simple connected graph embedded into a $2$-sphere with positive combinatorial curvature everywhere and minimum degree $\delta(G)\geq 3$, then $G$ is finite. Higuchi's conjecture was verified by Sun and Yu \cite{Sun-Yu04} for cubic planar graphs and resolved by DeVos and Mohar~\cite{DeVos-Mohar07}. In particular, DeVos and Mohar showed the following theorem.

\begin{theorem}[\cite{DeVos-Mohar07}]\label{thm:DM}
Suppose $G$ is a connected simple graph embedded into a $2$-dimensional topological manifold $\Omega$ without boundary and $G$ has minimum degree at least $3$. If $G$ has positive combinatorial curvature, then it is finite and $\Omega$ is homeomorphic to either a $2$-sphere or a projective plane. Moreover, if $G$ is not a prism, an antiprism, or one of their projective plane analogues, then $|V(G)| \leq 3444$.
\end{theorem}

The minimum possible constants for $|V(G)|$ in Theorem \ref{thm:DM} for $G$ embedded in a $2$-sphere and a projective plane, respectively, were studied in \cite{RBK05, Nicholson-Sneddon11, Chen-Chen08, Zhang08, Oh17}. In particular, Nicholson and Sneddon~\cite{Nicholson-Sneddon11} gave examples of positively (combinatorially) curved graphs with $208$ vertices embedded into a $2$-sphere. The upper bound on $|V(G)|$ was recently settled by Ghidelli \cite{Ghidelli17}. Planar graphs with nonnegative combinatorial curvature have also been extensively studied (see \cite{Hua-Su2019} and the references within). Tight upper bounds on certain graph classes with positive Lin--Lu--Yau curvature everywhere are also obtained in \cite{GLLY2023+}.

Recently, Lu and Wang \cite{Lu-Wang2023+} initiated the study of the order of planar graphs with positive \textit{Lin--Lu--Yau} curvature (\textit{LLY curvature} for short), which is defined on the vertex pairs of a graph. 
A graph $G$ is called \textit{positively LLY-curved} if it has positive Lin--Lu--Yau curvature on every vertex pair of $G$. 
For brevity of the introduction, we give the definition of the Lin--Lu--Yau curvature in Section \ref{sec:LLY-curvature}. In~\cite{Lu-Wang2023+}, Lu and Wang established an analogue of DeVos and Mohar's result in the context of Lin--Lu--Yau curvature. In particular, they showed the following two theorems. 

\begin{theorem}[\cite{Lu-Wang2023+}]
Let $G$ be a simple positively LLY-curved planar graph $G$ with $\delta(G) \geq 3$. Then the maximum degree $\Delta(G) \leq 17$.
\end{theorem}

\begin{theorem}[\cite{Lu-Wang2023+}]
If $G$ is a simple positively LLY-curved planar graph with minimum degree at least~$3$, then $G$ is finite. In particular, $|V(G)| \leq 17^{544}$.
\end{theorem}

The question of determining a sharp upper bound on the order of a positively LLY-curved planar graph is still open and seems hard. 
In this paper, we obtain sharp upper bounds on the maximum degree and order of a positively LLY-curved (maximal) outerplanar graph. A graph is \textit{outerplanar} if it admits a planar embedding such that all vertices lie on the outer face. A \emph{maximal outerplanar} graph is an edge maximal outerplanar graph. 
We first show a sharp upper bound on the maximum degree of a positively LLY-curved outerplanar graph.

\begin{theorem} \label{thm:maxdeg}
    Let $ G $ be a simple positively LLY-curved outerplanar graph with $ \delta(G) \geq 2 $. Then $ \Delta(G) \leq 9$ and the upper bound is sharp.
\end{theorem}

Using Theorem \ref{thm:maxdeg} and additional degree constraints shown in Section \ref{sec:degree_constraints},
we obtain a sharp upper bound on the order of positively LLY-curved maximal outerplanar graphs.

\begin{theorem}\label{thm:main}
    Let $G$ be a positively LLY-curved maximal outerplanar graph. Then $|V(G)| \leq 10$ and the upper bound is sharp.
\end{theorem}

The upper bounds in Theorem \ref{thm:maxdeg} and Theorem \ref{thm:main}  are both achieved by the fan graph on $10$ vertices (see Figure \ref{fig:fan}).

\begin{figure}[th]
    \centering
    \vspace{0.15cm}
    \includegraphics[scale=.3]{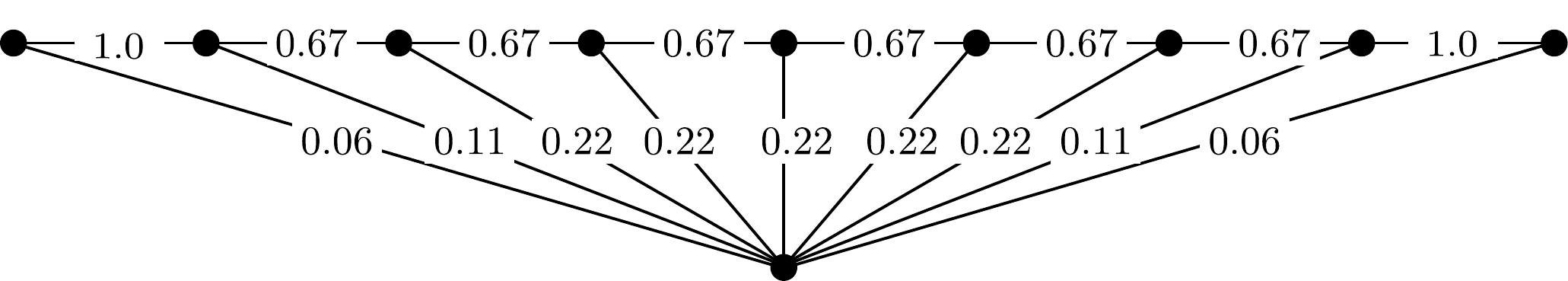}
    \vspace{0.15cm}
    \caption{Fan graph on 10 vertices with curvature of each edge computed by SageMath.}
    \label{fig:fan}
\end{figure}

For convenience, in the rest of the paper, we simply say a graph $G$ is positively curved if it is positively LLY-curved, and we simply denote $\kappa_{LLY}(x,y)$ by $\kappa(x,y)$.

\medskip

{\bf Notation and Terminology.} Given a graph $G$ and $x\in V(G)$, we use $N(x)$ to denote the neighborhood of $x$, i.e.,~$N(x) = \{y \in V(G) : xy \in E(G)\}$. Let $d_x$ denote the degree of $x$, i.e., $d_x = |N(x)|$. Let $N[x]:= N(x)\cup \{x\}$. Given $x,y\in V(G)$, let $N(x,y)$ denote the common neighborhood of $x$ and $y$, i.e., $N(x,y):=N(x)\cap N(y)$. Moreover, given $S\subseteq V(G)$, define $N(S) := \{x \in V(G) : xs \in E(G) \text{ for some } s \in S \}$ and $E(x,S) := \{ xs \in E(G): s \in S\}$.  Throughout the paper, for $xy \in E(G)$, we will consider the \emph{local configuration} $H_{xy} \subseteq G$, the induced subgraph of $G$ defined by $H_{xy} = G[N[x]\cup N[y]].$ 
For $x,y \in V(G)$, we use $d(x,y)$ to denote the number of edges in a shortest path between $x$ and $y$. Given an outerplanar graph $G$ and an edge $e$, we call $e$ an \emph{exterior edge} if it lies on the boundary of the outer face of $G$, and we call $e$ an \emph{interior edge} otherwise.

\section{Lin--Lu--Yau Ricci curvature}\label{sec:LLY-curvature}

In this section, we define the Lin--Lu--Yau Ricci curvature. We closely follow the notation of \cite{Lu-Wang2023+, Ollivier}. Let $m_1$ and $m_2$ be two probability distributions on $V(G)$. That is, for $i=1,2$, we have $m_i: V(G) \to [0,1]$ such that $\sum_{x \in V(G)} m_i(x) = 1$. A \emph{coupling} between $m_1$ and $m_2$, is a map $A: V(G) \times V(G) \to [0,1]$ with finite support such that $\sum_{y \in V(G)} A(x,y) = m_1(x)$ and $\sum_{x \in V(G)} A(x,y) = m_2(y)$.
The \emph{transportation distance} between $m_1$ and $m_2$ is defined as
\begin{equation}
 W(m_1, m_2) = \inf_A \sum_{x,y \in V(G)} A(x,y)\,d(x,y),
\end{equation}
where the infimum is taken over all couplings $A$ between $m_1$ and $m_2$. By the duality theorem of a linear optimization problem, the transportation distance can also be written as
\begin{equation}
W(m_1, m_2) = \sup_{f\in Lip(1)} \dss_{x\in V(G)} f(x) \lp m_1(x)-m_2(x)\rp,
\end{equation}
where the supremum is taken over all $1$-Lipschitz functions $f$.
A random walk $m$ on a graph $G=(V,E)$ is defined as a family of probability measures $\{m_v(\cdot)\}_{v\in V(G)}$ such that $m_v(u) = 0$ for all $uv \notin E(G)$. It follows that  $m_v(u) \geq 0$ for all $v,u\in V(G)$ and $\sum_{u\in N[v]} m_v(u) = 1$. 
 For $v \in V(G)$, consider the~$\alpha$-lazy random walk $\{m_v^{\alpha}\}_{v\in V(G)}$ defined as
\begin{align*}
 m^\alpha_v(u) = \begin{cases}
     \alpha &\text{\qquad  if   } u=v,\\
     \frac{1-\alpha}{d_v} &\text{\qquad  if   } u \in N(v), \\
     0 &\text{\qquad  otherwise}.
 \end{cases}
\end{align*}
In \cite{LLY}, Lin, Lu and Yau defined the Ricci curvature of graphs based on the $\alpha$-lazy random walk as $\alpha$ goes to $1$. More precisely,
for any $x\neq y \in V$, they defined the $\alpha$-Ricci-curvature $\kappa_{\alpha}(x,y)$ to be 
\begin{equation}
\kappa_{\alpha}(x,y) = 1 - \frac{W(m_x^{\alpha}, m_y^{\alpha})}{d(x,y)}
\end{equation} and the Lin--Lu--Yau Ricci curvature $\kappa_{\textrm{LLY}}$ of the vertex pair $(x,y)$ in $G$ to be 
\begin{equation}
\kappa_{\textrm{LLY}}(x,y) = \ds\lim_{\alpha \to 1} \frac{\kappa_{\alpha}(x,y)}{(1-\alpha)}.
\end{equation}
Recall that a graph $G$ is called positively LLY-curved if it has positive Lin--Lu--Yau curvature on every vertex pair of $G$. It was shown \cite{LLY} that if $\kappa_{\textrm{LLY}}(x,y)\geq \kappa_0$ for every edge $xy\in E(G)$, then $\kappa_{\textrm{LLY}}(x,y)\geq \kappa_0$ for any pair of vertices $(x,y)$. 

Recently, M\"{u}nch and Wojciechowski \cite{MW} gave a limit-free formulation of the Lin--Lu--Yau Ricci curvature using \textit{graph Laplacian}. Given a graph $G = (V,E)$, the graph Laplacian $\Delta$ is defined as
\begin{equation}
\Delta f(v)=\frac{1}{d_v}\sum\limits_{u\in N(x)} (f(u)-f(v)). 
\end{equation}
The Lin--Lu--Yau curvature can be alternatively expressed \cite{MW} as
\begin{equation}\label{eq:LLY-curvature}
    \kappa_{LLY} (x,y) = \inf_{\substack{f \in Lip(1),\\\nabla_{yx}f = 1}} \nabla_{xy} \Delta f,
\end{equation}
where $\nabla_{vu} f = \frac{f(v) - f(u)}{d(v,u)}$
is the \emph{gradient} function.
Following Lemma 2.2 of \cite{BM2015}, it suffices to optimize over all integer-valued $1$-Lipschitz functions $f$. An integer-valued $1$-Lipschitz function $f: V(G)\to \mathbb{Z}$ that attains the infimum in Equation \eqref{eq:LLY-curvature} is called 
\emph{optimal for $(x,y)$}.

\section{Proof of Theorem \ref{thm:maxdeg}.}
In this section, we show a sharp upper bound on the maximum degree of a positively curved outerplanar graph. We first recall the following lemma of Lu and Wang \cite{Lu-Wang2023+}.

\begin{lemma}[\cite{Lu-Wang2023+}]\label{lem:genbound}
Let $G$ be a simple positively curved planar graph and $xy \in E(G)$ be an arbitrary edge with $d_x \leq d_y$. Suppose $S \subseteq N(x) \setminus \{y\}$, $|S| = s$, and $|S \cap N(y)| = k$. Then
\begin{align*}
    |N(S) \cap N(y)| > \frac{s}{d_x}d_y - (k+1 + |N(x,y)|) + |N(S) \cap N(x,y)|.
\end{align*}
\end{lemma}
    
     When $S = N(x) \setminus \{y\}$ in the equation above, we have $s=d_x -1$ and $k = |N(x,y)|$. We then obtain the following immediate corollary of \lemref{genbound}.
     
    \begin{corollary}\label{cor:1}
		Let $G$ be a simple positively curved planar graph and $xy\in E(G)$ be an arbitrary edge with $d_x \leq d_y$. Let $ S = N(x) \setminus \{y\} $. Then 
		\[
		 \ab{N(S) \cap N(y)} > \pa{1 - \frac{1}{d_x}}d_y - (2\ab{N(x,y)}+1) + \ab{N(S)\cap N(x,y)}.
		\]
	\end{corollary}

	\begin{proof}[Proof of Theorem \ref{thm:maxdeg}]
		Suppose, towards a contradiction, that $ \Delta(G) \geq 10$.
  Let $ x,y \in V(G) $ such that~$d_y = \Delta(G) $ and~$xy \in E(G) $ is an exterior edge. 
  Then $d_x \ge \delta(G) \geq 2$. 
  Let $S = N(x) \setminus \{y\}$ and let $ T\deq N(S) \cap N(y)$. Since $xy$ is an exterior edge, the common neighborhood $N(x,y)$ may either be empty, or contain one vertex. We consider these cases separately.\vspace{5pt}\\
		\textit{Case $1$:}  
		Suppose $ \ab{N(x,y)} = 0 $. Using \corref{1},
		\begin{align*}
			\ab{T} &> \pa{1 - \frac{1}{d_x}}d_y - 1 + \ab{N(S)\cap N(x,y)}\\
			&\geq \frac{1}{2}\cdot 10 - 1 + \ab{N(S) \cap N(x,y)}\\
			&\geq 4.
		\end{align*}
		  Hence $ \ab{T} \geq 5 $. However, note that contracting all edges in $ E(x,S) $ yields a $ K_{2, 3}$ minor in $G$, contradicting the outerplanarity of $G$.\vspace{5pt}\\
		\textit{Case $2$:}
		Suppose $ \ab{N(x,y)} = 1 $ and let $z$ be the only common neighbor in $N(x,y)$. If $ d_x = 2 $, then $ |S| = |N(x) \setminus \{y\}|=1 $ and
		\begin{align*}
			\ab{T} &> \pa{1 - \frac{1}{d_x}}d_y - 3 + \ab{N(S)\cap N(x,y)}\\
			&\geq \frac{1}{2}\cdot 10 - 3 + \ab{N(S) \cap N(x,y)}\\
			&\geq 2.
		\end{align*}
		  Hence, $ \ab{T} \geq 3 $. Therefore, the only vertex in $ S $ (which is $z$) is adjacent to at least $3$ neighbors of $ y $, yielding a $ K_{2,3} $ minor and contradicting the outerplanarity of $G$. If $ d_x \geq 3$, then
		\begin{align*}
			\ab{T} &> \pa{1 - \frac{1}{d_x}}d_y - 3 + \ab{N(S)\cap N(x,y)}\\
			&\geq \frac{2}{3}\cdot 10 - 3 + \ab{N(S) \cap N(x,y)}\\
			&\geq \frac{11}{3}.
		\end{align*}
  Hence, $ \ab{T} \geq 4 $.
  Since $d_y = \Delta(G) \ge 10$ and  $d_x \geq 3 $, 
  the sets $ N(y)\setminus \{x,z\}$ and $ S \setminus \{z\} $ must be nonempty. 
  For any $ v \in N(y)\setminus\{x, z\}$ and $ w \in S \setminus \{z\}$, 
  if $ vw \in E(G) $, then either $z$ lies in the interior of the cycle $wxyv$ which contradicts that $G$ is outerplanar, or $xy$ lies in the interior of the cycle $wxyv$ which contradicts that $xy$ is an exterior edge.
    Thus, we have that for any $ v \in N(y)\setminus \{x,z\}$ and $ w \in S \setminus \{z\}$, $vw \notin E(G)$.

   Therefore, any $v\in T\setminus\{x,z\}$ must be adjacent to $z$. Since $|T\setminus \{x,z\}|\geq 2$ and $x\in N(z),$ $z$ is adjacent to at least three vertices in~$N(y)$, yielding a $ K_{2,3} $ minor together with $y$ and $N(z,y)$, contradicting the outerplanarity of $G$.
\end{proof}

\section{Exact Curvature via Local Configuration}\label{sec:degree_constraints}
In this section, we characterize the possible degree pairs of an edge with positive Lin--Lu--Yau curvature in a maximal outerplanar graph by determining the optimal $1$-Lipschitz functions for every edge $xy\in E(G)$. Recall that  
\begin{equation}
    \kappa(x,y) = \inf_{\substack{f \in Lip(1),\\\nabla_{yx}f = 1}} \nabla_{xy} \Delta f, \;\;\;\textrm{where $\nabla_{xy} f = \frac{f(x) - f(y)}{d(x,y)}$.}
\end{equation}
Moreover, recall that by Lemma 2.2 of \cite{BM2015}, it suffices to consider integer-valued optimal functions $f_0$. Thus, since $\nabla_{yx}f_0 =1$, we may assume without loss of generality that $f_0(x)=0$ and $f_0(y) =1$.
  Then, due to the $1$-Lipschitz condition, $\text{Im} (f_0|_{N(x)}) \subseteq \{ -1, 0, 1\} $ and $\text{Im} (f_0|_{N(y)}) \subseteq \{ 0, 1, 2\}$ where Im denotes the image of a function. 

\begin{lemma}
\label{lem:exterior}
    Let $G$ be a maximal outerplanar graph, and suppose $xy \in E(G)$ is an exterior edge such that $d_x \leq d_y$. Then, $\kappa_{LLY} (x,y) > 0 $ if and only if
    \[(d_x, d_y) \in \left\{ (2,2), (2,3), (2,4), (2,5), (2,6), (2,7), (2,8), (2,9), (3,3), (3,4) \right\}.\]
\end{lemma}

\begin{table}[ht]
\begin{adjustbox}{width=0.895\columnwidth,center}
\begin{tabular}{|c|c|c|c|c|l|l|}
\hline
\begin{minipage}[c][0.25in][c]{1in} \centering $H_{xy}$ \end{minipage}& $\delta_{xz}$ & $d_x$ & $\delta_{yz}$ & $d_y$ & $\kappa (x,y)$ & Positive degree pairs  $(d_x, d_y)$ \\
\hline
\cT{
 \includegraphics[width=1in]{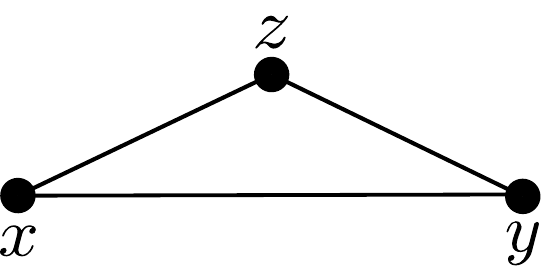}} & 0 & 2 & 0 & 2 & $ 3/d_x + 4/d_y - 2$ & $\{(2,2)\}$\\
\cT{\includegraphics[width=1in]{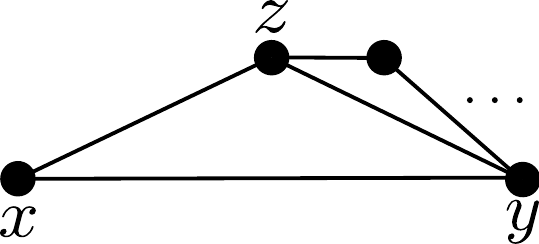}} & 0 & 2 & 1 &$\ge 3$ &  $
    \begin{cases}
        4/d_x + 3/d_y - 2 & \text{if } d_y < 2d_x\\
        3/d_x + 5/d_y - 2 & \text{if } 2d_x \leq d_y
    \end{cases}$ & $\{(2,3), (2,4), (2,5), (2,6), (2,7), (2,8), (2,9)\} $ \\
\cT{\includegraphics[width=1in]{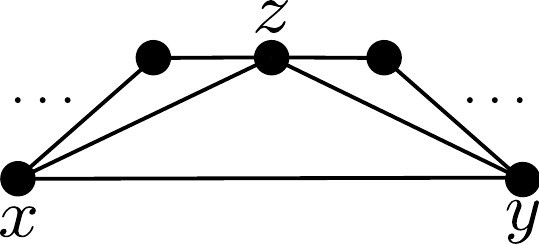}} & 1 & $\ge 3$ & 1& $\ge d_x$  & $ 3/d_x + 5/d_y - 2$ & $\{(3,3), (3,4)\}$ \\
\hline
\end{tabular}
\end{adjustbox}
\caption{Degree pairs for positively curved exterior edges $xy$ and local configuration $H_{xy}$.}
\label{tab:exterior}
\end{table}

\begin{proof}[Proof of Lemma \ref{lem:exterior}]
    Let $xy$ be an exterior edge of $G$. We will deduce optimal functions $f_0$ for the edge $xy$ in all possible local configurations around $xy$, that is, $f_0$ such that
\begin{equation}\label{eq:lem2}
\kappa(x,y) = \inf_{\substack{f \in \Lip(1),\\\nabla_{yx}f = 1}} \nabla_{xy} \Delta f = \nabla_{xy} \Delta f_0.
\end{equation}
Without loss of generality, we assign $f_0(x) =0$ and $f_0(y)=1$. Since $xy$ is an exterior edge and~$G$ is edge maximal, $N(x) \cap N(y) = \{z\}$ for some $z \in V(G)$. Then, since $f_0 \in \Lip(1)$, we must have $f_0(z) \in \{0, 1\}$. 
Let $\delta_{xz} = |N(x,z)\setminus\{y\}|$ and $\delta_{yz} = |N(y,z)\setminus\{x\}|$. Note that, since $G$ is outerplanar, $\delta_{xz}, \delta_{yz} \in \{ 0, 1\}$. If $\delta_{xz} = 1$, let $N(x,z)\setminus\{y\} = \{v_{xz}\}$. Similarly, if $\delta_{yz} = 1$, let $N(y,z)\setminus\{x\} = \{v_{yz}\}$.
Since $f_0(x)=0$ and $f_0(y)=1$, we then obtain that
\begin{align}
    \nabla_{xy} \Delta f_0 &= \frac{1}{d_x} \sum_{v \in N(x)} (f_0(v) - f_0(x)) - \frac{1}{d_y} \sum_{v \in N(y)} (f_0(v) - f_0(y))
    \nonumber \\
    &= \frac{1}{d_x} \sum_{v \in N(x)} f_0(v) + \frac{1}{d_y} \sum_{v \in N(y)} (1-f_0(v)).
    \label{eqn:expansion}
\end{align}

Observe that since $G$ is outerplanar, $xy$ is an exterior edge and $\{z\}=N(x,y)$, we have that the distance from any vertex in $N(x) \setminus\{v_{xz},y,z\}$ to any vertex in $N(y) \setminus \{v_{yz},x,z\}$ is $3$. Since $f_0$ is an optimal function for $xy$, we then obtain that $f_0(v) = -1 $ for all $v \in N(x) \setminus\{v_{xz},y,z\}$ and $f_0(v) = 2$ for all $v \in N(y) \setminus\{v_{yz},x,z\}$. 
Observe that if $\delta_{xz}=1$, i.e., $v_{xz}$ exists, then by the $1$-Lipschitz and integer-valued property of $f_0$, we have that $f_0(v_{xz})\in \{f_0(z),f_0(z)-1\}$. Now since $f_0$ is optimal, it is easy to see that $f_0(v_{xz})=f_0(z)-1$ so that $\Delta f_0(x)$ is minimal in \eqref{eq:lem2}. Similarly, if $\delta_{yz}=1$, then $f_0(v_{yz})=f_0(z)+1$.
Equation \eqref{eqn:expansion} can then be re-written as
\begin{align}
     \nabla_{xy} \Delta f_0 &=
     \frac{1}{d_x}\left[ \delta_{xz}f_0(v_{xz}) + f_0(z) + 1 - (d_x - 2 - \delta_{xz})\right] \nonumber \\
    &\quad\quad\quad + \frac{1}{d_y}\left[\delta_{yz}(1-f_0(v_{yz})) + (1-f_0(z)) + 1 - (d_y - 2 - \delta_{yz})\right]\nonumber \\
    &= \frac{1}{d_x}\left[ \delta_{xz}(f_0(z) - 1) + f_0(z) + 1 - (d_x - 2 - \delta_{xz})\right] \nonumber \\
    &\quad\quad\quad + \frac{1}{d_y}\left[ -\delta_{yz}f_0(z) + (1-f_0(z)) + 1 - (d_y - 2 - \delta_{yz})\right]\nonumber \\
    &= \left(\frac{1 + \delta_{xz}}{d_x} - \frac{1 + \delta_{yz}}{d_y}\right)f_0(z) + \frac{3}{d_x} + \frac{4 + \delta_{yz}}{d_y} - 2. \label{eqn:expansion2}
\end{align}

Note that $\frac{1 + \delta_{xz}}{d_x} - \frac{1 + \delta_{yz}}{d_y} < 0$ if and only if $\delta_{xz} = 0, \delta_{yz} = 1$ and $d_y < 2d_x$. 

\textit{Case 1}: $\frac{1 + \delta_{xz}}{d_x} - \frac{1 + \delta_{yz}}{d_y} < 0$. Then it is clear that $\nabla_{xy} \Delta f_0$ is minimized when $f_0(z)=1$. Moreover, since in this case $\delta_{xz}=0$, $f_0$ still satisfies the $1$-Lipschitz condition. 

\medskip

\textit{Case 2}:  $\frac{1 + \delta_{xz}}{d_x} - \frac{1 + \delta_{yz}}{d_y} \geq 0$. Then it is clear that $\nabla_{xy} \Delta f_0$ is minimized when $f_0(z)=0$. Suppose $\delta_{xz} = 1$ and $\delta_{yz} = 0$, then $d_x\geq 3$. Since $G$ is maximally outerplanar, $yz$ must be an exterior edge, which implies $d_y=2$, a contradiction of $d_x \leq d_y$. Hence, it follows that $(\delta_{xz},\delta_{yz}) \in \{(0,0), (0,1), (1,1)\}$. 
We list each local configuration based on $\delta_{xz}$ and $\delta_{yz}$ in \tabref{exterior} and characterize the possible degree pairs when $\kappa(x,y)>0$. It is not hard to check that $f_0$ can be made $1$-Lipschitz in each case. 
\end{proof}

\begin{lemma}
\label{lem:interior}
    Let $G$ be a maximal outerplanar graph, and suppose $xy \in E(G)$ is an interior edge such that $d_x \leq d_y$. A complete characterization of degree pairs $(d_x, d_y)$, depending upon the local configuration, such that $\kappa_{LLY} (x,y) > 0 $ is given by \tabref{interior}.
\end{lemma}
    \begin{table}[ht]
\begin{adjustbox}{width=0.895\columnwidth,center}

\begin{tabular}{|c|c|c|c|c|c|c|l|l|}
\hline
\begin{minipage}[c][0.25in][c]{1in} \centering $H_{xy}$ \end{minipage}&$\delta_{xw}$ & $\delta_{xz}$ & $d_x$& $\delta_{yw}$ & $\delta_{yz}$ &$d_y$& $\kappa (x,y)$ & Positive degree pairs  $(d_x, d_y)$\\
\hline
\cT{\includegraphics[width=1in]{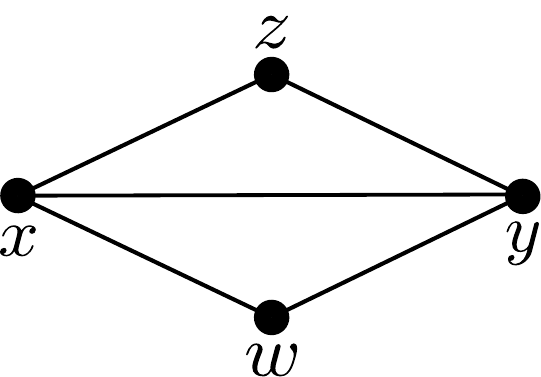}}&0 & 0 &3& 0 & 0 &3& $4/d_x + 6/d_y - 2$ & $\{(3,3)\}$\\
\cT{\includegraphics[width=1in]{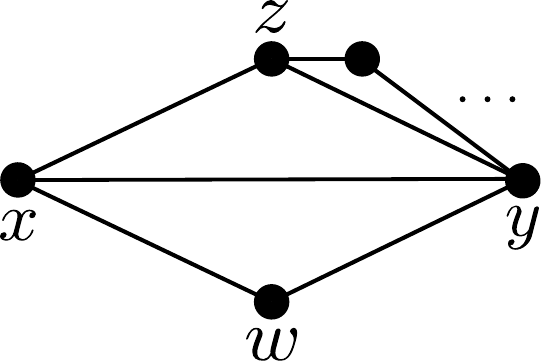}}&0 & 0 &3& 0 & 1 &$\ge 4$&  $\begin{cases}
        5/d_x + 5/d_y - 2 & \text{if } d_y < 2d_x\\
        4/d_x + 7/d_y - 2 & \text{if } 2d_x \leq d_y
    \end{cases}$ & $\{(3,4),(3,5),(3,6),(3,7),(3,8),(3,9)\}$\\
\cT{\includegraphics[width=1in]{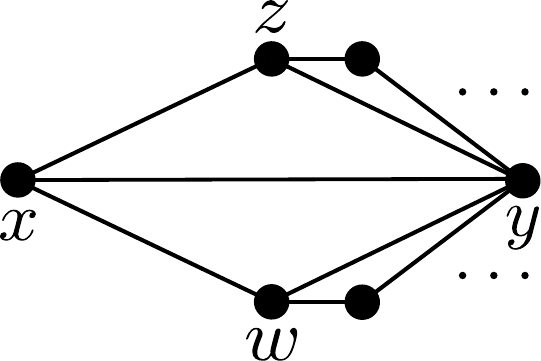}}&0 & 0 & 3& 1 & 1 & $\ge 5$&  $\begin{cases}
        6/d_x + 4/d_y - 2 & \text{if } d_y < 2d_x\\
        4/d_x + 8/d_y - 2 & \text{if } 2d_x \leq d_y
    \end{cases}$ & $\{(3,5), (3,6),(3,7),(3,8),(3,9)\}$\\
\cT{\includegraphics[width=1in]{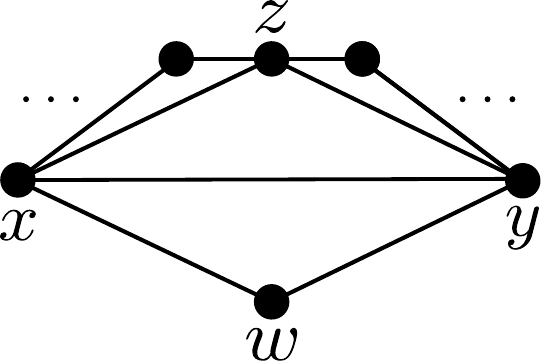}}&0 & 1 & $\ge 4$&0 & 1 &$\ge d_x$& $4/d_x + 7/d_y - 2$ & $\{(4,4),(4,5),(4,6),(5,5)\}$\\
\cT{\includegraphics[width=1in]{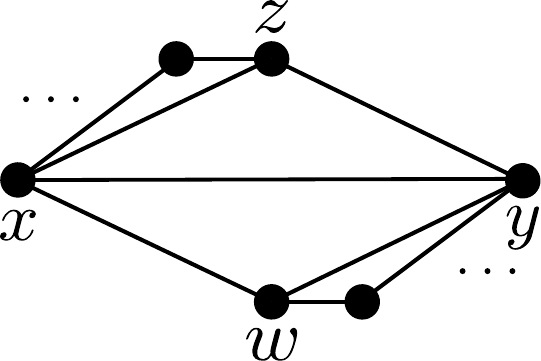}}&0 & 1 &$\ge 4$& 1 & 0 & $\ge d_x$&  $\begin{cases}
        5/d_x + 5/d_y - 2 & \text{if } d_y < 2d_x\\
        4/d_x + 7/d_y - 2 & \text{if } 2d_x \leq d_y
    \end{cases}$ & $\{(4,4),(4,5),(4,6)\}$\\
\cT{\includegraphics[width=1in]{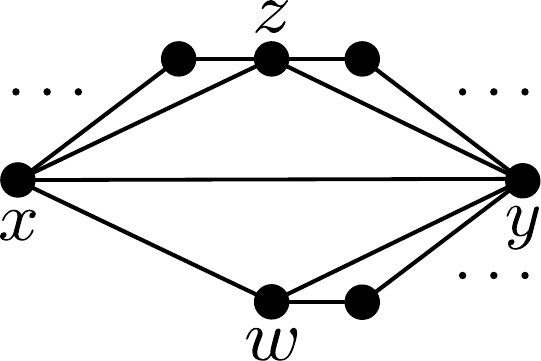}}&0 & 1 &$\ge$ 4 & 1 & 1 & $\ge \max\{5, d_x\}$&  $\begin{cases}
        5/d_x + 6/d_y - 2 & \text{if } d_y < 2d_x\\
        4/d_x + 8/d_y - 2 & \text{if } 2d_x \leq d_y
    \end{cases}$ & $\{(4,5),(4,6),(4,7),(5,5)\}$\\
\cT{\includegraphics[width=1in]{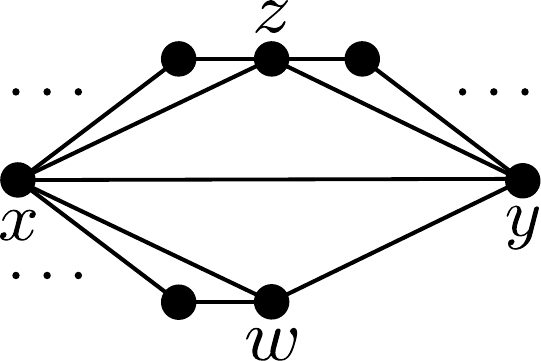}}&1 & 1 & $\ge$ 5 &0 & 1 &$\ge d_x$& $4/d_x + 7/d_y - 2$ & $\{(5,5)\}$\\
\cT{\includegraphics[width=1in]{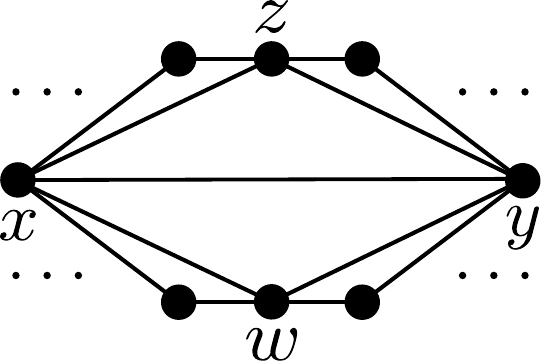}} &1 & 1 &$\ge$ 5 & 1 & 1 &$\ge d_x$& $4/d_x + 8/d_y - 2$ & $\{(5,5),(5,6)\}$\\
\hline
\end{tabular}
\end{adjustbox}
\caption{Positive degree pairs for interior edges $xy$ based on the local configuration $H_{xy}$.}
\label{tab:interior}
\end{table}

\begin{proof}[Proof of Lemma \ref{lem:interior}]
 Let $xy$ be an interior edge of $G$. We will deduce optimal functions $f_0$ for the edge $xy$ in all possible local configurations around $xy$. Similar to before, without loss of generality, we assume that $f_0(x) =0$ and $f_0(y)=1$.

Since $xy$ is an interior edge and $G$ is edge maximal, $N(x) \cap N(y) = \{w,z\}$ for some $w,z \in V(G)$ and by the $1$-Lipschitz condition we have $f_0(w), f_0(z) \in \{0, 1\}$. Similar to before, Equation \eqref{eqn:expansion} still holds, i.e., 
\begin{align}
    \nabla_{xy} \Delta f_0 &= \frac{1}{d_x} \sum_{v \in N(x)} (f_0(v) - f_0(x)) - \frac{1}{d_y} \sum_{v \in N(y)} (f_0(v) - f_0(y))\nonumber\\
    &= \frac{1}{d_x} \sum_{v \in N(x)} f_0(v) + \frac{1}{d_y} \sum_{v \in N(y)} (1-f_0(v)).
    \label{eqn:secondexpansion}
\end{align}

Define $\delta_{xw} = |N(x,w)\setminus\{y\}|$, $\delta_{xz} = |N(x,z)\setminus\{y\}|$, $\delta_{yw} = |N(y,w)\setminus\{x\}|$ and $\delta_{yz} = |N(y,z)\setminus\{x\}|$. 
Since $G$ is outerplanar, $\delta_{xw}, \delta_{xz}, \delta_{yw}, \delta_{yz} \in \{ 0, 1\}$. For $a\in \{x,y\}$ and $b\in \{w,z\}$, if $\delta_{ab}=1$, then let $N(a,b)\backslash \{x,y\} = \{v_{ab}\}$.

Observe that in \eqref{eqn:secondexpansion}, $\nabla_{xy} \Delta f_0$ is minimized when $f_0(v)$ is as small as possible for $v\in N(x)$ and as big as possible for $v\in N(y)$ while keeping $f_0$ $1$-Lipschitz. Note that since $G$ is outerplanar, $\{w,x\}$ separates $v_{xw}$ (if it exists) with vertices in $N(y)\backslash \{x,w\}$. Similar to Lemma \ref{lem:exterior}, if $\delta_{xw}=1$, then $f_0(v_{xw}) = f_0(w)-1$ due to the optimality of $f_0$. Similarly, $f_0(v_{xz}) = f_0(z) - 1$, $f_0(v_{yw}) = f_0(w)+1$, $f_0(v_{yz}) = f_0(z) + 1$ if $v_{xz}, v_{yw}, v_{yz}$ exist respectively. 

Moreover, observe that the distance from any vertex in $N(x) \setminus\{v_{xw}, w, y, z, v_{xz}\}$ to any vertex $N(y) \setminus \{v_{yw},w,x,z, v_{yz}\}$ is $3$. Therefore, since $f_0$ is optimal, we obtain that $f_0(v) = -1 $ for all $v \in N(x) \setminus\{v_{xw}, w, y, z, v_{xz}\}$ and $f_0(v) = 2$ for all $v \in N(y) \setminus\{v_{yw},w,x,z, v_{yz}\}$.
Then Equation \eqref{eqn:secondexpansion} becomes
\begin{align*}
    \nabla_{xy} \Delta f_0 &= \frac{1}{d_x}\left[ \delta_{xw}(f_0(w) - 1) + f_0(w) + 1 + f_0(z) + \delta_{xz}(f_0(z) - 1) - (d_x - 3 - \delta_{xw} - \delta_{xz})\right] \\
    &\quad\quad\quad + \frac{1}{d_y}\left[ -\delta_{yw}f_0(w) + (1-f_0(w)) + 1 + (1-f_0(z)) - \delta_{yz}f_0(z) - (d_y - 3 - \delta_{yw} - \delta_{yz})\right]\\
    &= \left(\frac{1 + \delta_{xw}}{d_x} - \frac{1 + \delta_{yw}}{d_y}\right)f_0(w) + \left(\frac{1 + \delta_{xz}}{d_x} - \frac{1 + \delta_{yz}}{d_y}\right)f_0(z) + \frac{4}{d_x} + \frac{6 + \delta_{yw} + \delta_{yz}}{d_y} - 2.
\end{align*}

Note that $\frac{1 + \delta_{xw}}{d_x} - \frac{1 + \delta_{yw}}{d_y} < 0$ if and only if $\delta_{xw} = 0, \delta_{yw} = 1$ and $d_y < 2d_x$. In that case, since $f_0$ is optimal, it must be the case that $f_0(w) = 1$. On the other hand, if $\frac{1 + \delta_{xw}}{d_x} - \frac{1 + \delta_{yw}}{d_y}$ is nonnegative, $\nabla_{xy} \Delta f_0$ is minimized when $f_0(w) = 0$.

Similarly, note that  $\frac{1 + \delta_{xz}}{d_x} - \frac{1 + \delta_{yz}}{d_y} < 0$ if and only if $\delta_{xz} = 0, \delta_{yz} = 1$ and $d_y < 2d_x$. In that case, it must happen that $f_0(z) = 1$ as before. On the other hand, if $\frac{1 + \delta_{xz}}{d_x} - \frac{1 + \delta_{yz}}{d_y} \geq 0$, we may assume $f_0(z) = 0$.  Note that the case when $\delta_{xw} + \delta_{xz} \geq 1$ and $\delta_{yw} = \delta_{yz} = 0$ is impossible since $d_x \leq d_y$. We list each local configuration (up to symmetry) for $\delta_{xw}, \delta_{xz}, \delta_{yw}$ and $\delta_{yz}$ in \tabref{interior} and compute the possible degree pairs when $\kappa(x,y)>0$.
\end{proof}

\section{Bounding the Order}
In this section, we bound the order of a positively curved maximal outerplanar graph~$G$. Note that since $G$ is maximally outerplanar, if $|V(G)|\geq 3$, then $G$ is connected and $\delta(G) \geq 2$. Given two graphs $A$ and $B$, let the \textit{join} of $A$ and $B$, denoted by $A\vee B$, be the graph obtained from $A \cup B$ by connecting each vertex in $A$ with each vertex in $B$. We begin with the following structural lemma. 
\begin{lemma}
\label{lem:fansub}
Let $G$ be a maximal outerplanar graph and suppose $x \in V(G)$. Then, $G[N[x]] \cong K_1 \vee P_{d_x}$.
\end{lemma}
\begin{proof}
 Since $G$ is outerplanar, $G[N(x)]$ must be a disjoint union of paths. Moreover, since $G$ is edge-maximal, $G[N(x)]$ must be a single path (otherwise we could add an edge to merge two paths into one without violating the outerplanarity condition). Thus $G[N[x]] \cong K_1 \vee P_{d_x}$.
\end{proof}

\begin{proof}[Proof of Theorem \ref{thm:main}]
    We will show Theorem \ref{thm:main} by induction on $|V(G)|$. For the base case, we consider~$|V(G)| = 11$. Using SageMath, we verify that all $228$ maximal outerplanar graphs on $11$ vertices are non-positively curved.\footnote{The SageMath code used to verify is available at: \href{https://github.com/ghbrooks28/outerplanar-LLY/}{https://github.com/ghbrooks28/outerplanar-LLY/}.} 
    
    Now consider $G$, a maximal outerplanar graph on $n\geq 12$ vertices. Assume, towards a contradiction, that $G$ is positively curved. Let $w \in V(G)$ such that $d_w = 2$ (such $w$ exists since every maximal outerplanar graph contains a vertex of degree $2$). Note that~$G - w$ is maximally outerplanar. It follows from the induction hypothesis that~$G-w$ is not positively curved, thus there exists some $xy\in E(G)$ with $\kappa(x,y)\leq 0$. Without loss of generality, assume that $d_x\leq d_y$.
    We write $d_v$ for the degree of $v$ in $G$, and we write $\Tilde{d}_v$ for the degree of $v$ in $G-w$. 
    Since $xy$ is non-positively curved in $G-w$, \lemrefTwo{exterior}{interior} exactly characterize the possible values of $(\Tilde{d}_x, \Tilde{d}_y)$; we call degree pairs corresponding to non-positively curved edges \emph{bad} degree pairs and degree pairs corresponding to positively curved edges \emph{good} degree pairs.

    First, suppose $ w\not \in N(x) \cup N(y)$, so that $(d_x , d_y) = (\Tilde{d}_x, \Tilde{d}_y)$. Since  $(\Tilde{d}_x, \Tilde{d}_y)$ is a bad degree pair and the degrees of $x$ and $y$ are the same in $G$ and $G-w$, we have that $(d_x, d_y)$ is a bad degree pair, a contradiction.
    Next, suppose $w \in N(x)\setminus N(y)$ so that $(d_x, d_y) = (\Tilde{d}_x + 1, \Tilde{d}_y)$. Notice that since $w \in N(x)\setminus N(y)$, $xy$ is an exterior (or interior) edge in $G$ if and only if it is an exterior (respectively, interior) edge in $G-w$. Examining \tabref{exterior} and \tabref{interior}, we see that, given a degree pair $(a,b)$ that is bad for exterior (or interior) edges, the degree pair $(a+1, b)$ is also bad for exterior (respectively, interior) edges.
    Thus, $(d_x, d_y)$ is a bad degree pair in $G$, contradicting our assumption that $G$ is positively curved. The case when $w \in N(y) \setminus N(x)$ is similar.

        The remainder of the proof addresses the final case, where $w \in N(x,y)$, so that $(d_x , d_y) = (\Tilde{d}_x +1, \Tilde{d}_y+1)$. Note that in this case, $xy$ has to be exterior edge of $G-w$. Since $xy$ is non-positively curved in $G-w$, we obtain from \lemref{exterior} that 
        \begin{equation}\label{eq:exterior_tilde}
            (\Tilde{d}_x, \Tilde{d}_y) \not\in \{(2,2),(2,3),(2,4),(2,5),(2,6),(2,7),(2,8),(2,9),(3,3),(3,4)\}.
        \end{equation} 

    Note that by Theorem \ref{thm:maxdeg}, $\Delta(G)\leq 9$. It follows that $\max\{\Tilde{d}_y, \Tilde{d}_x\} \leq 8$.
    By Equation \eqref{eq:exterior_tilde}, it suffices to consider $xy$ such that $\Tilde{d}_x \geq 3$ and $\Tilde{d}_y \geq 4$, where $\Tilde{d}_y > 4$ if $\Tilde{d}_x = 3$. Then by Lemma \ref{lem:interior}, $(d_x, d_y)=(\Tilde{d}_x +1, \Tilde{d}_y+1)$ is a bad degree pair (and so $G$ would be non-positively curved) for all values except $(d_x, d_y) \in \{ (4,6), (5,5), (4,7), (5,6)\}$. However, note that when $(d_x,d_y) = (4,7)$ or $(5,6)$, $\max\{\delta_{xw},\delta_{yw}\}=1$ (see \tabref{interior} in Lemma \ref{lem:interior}), which contradicts that $d_w=2$. Hence we are left with two cases when $(d_x, d_y) \in \{ (4,6), (5,5)\}$. We will examine the local configuration $H_{xy}$ below.

    \medskip
    
    \textit{Case 1:} Suppose $(d_x, d_y) = (4, 6).$ Let $C \subseteq G$ be the outer face cycle in $G$. Using \lemref{fansub}, we have the local configuration $H_{xy}$, labeled as shown in \figref{case1}. Let $x' \in N_C(x) \setminus \{w\}$ and $y',v,u\in N_C(y)\setminus \{w\}$ such that $y'vuzxw$ is the path induced by $N(y)$. Since $d_y = 6,$ $yy'$ is an exterior edge, and $yz, yu, yv$ are interior edges. Note that $G$ is assumed to be positively curved. Thus the degree pairs of $yy', yz, yu$ and $yv$ all have to be good. By Lemma \ref{lem:exterior} and Lemma \ref{lem:interior}, we have $d_{y'} = 2$, $d_z \in \{4,5\},$ and $d_u, d_v \in \{3,4,5\}.$ Since $d_x = 4,$ and $xx'$ is an exterior edge, $d_{x'} \in \{ 2,3\}$. 
    
    We first claim that $d_{x'}=2$. Otherwise, if $d_{x'} = 3$, then there exists $x'' \in V(G)\setminus\{x\}$ such that $x'' \in N(x',z),$ which implies $d_z \geq 5$. By Lemma \ref{lem:exterior} and Lemma \ref{lem:interior}, $zu$ must then be an interior edge, which implies that $d_z\geq 6$. But now $(d_y, d_z)$ is a bad degree pair, contradicting that $yz$ is positively curved. Hence $d_{x'}=2$. 
    
    But now if we repeat the argument by considering $G-x'$ instead of $G-w$, we obtain contradictions unless $(d_x, d_z) \in \{ (4,6), (5,5)\}$, which is impossible in this case since $d_x=4$ and $d_z < 6$. 
    
    \begin{figure}[h]
    \centering
      \includegraphics[width=.15\linewidth]{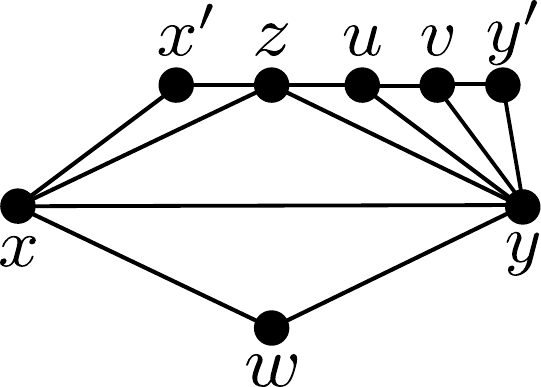}
      \caption{$H_{xy}$ in Case 1.}
      \label{fig:case1}
    \end{figure}
    
    
    \textit{Case 2:} Suppose $(d_x, d_y) = (5, 5).$ Let $C \subseteq G$ be the outer face cycle. Let $x',u\in N_C(x) \setminus \{w\}$ such that $x'uzyw$ is the path induced by $N(x)$ (see \figref{case2}). Since $d_x= 5,$ and $xx'$ is an exterior edge, it follows from Lemma \ref{lem:exterior} that $d_{x'}= 2$.
    Now if we repeat the argument by considering $G-x'$ instead of $G-w$, we obtain contradictions unless $(d_x,d_u)= (5,5)$. However, this is impossible since by the configurations in Lemma \ref{lem:interior}, if $(d_x,d_u)= (5,5)$, then $d_{x'}>2$, giving a contradiction.  
    \begin{figure}[h!]
    \centering
      \includegraphics[width=.15\linewidth]{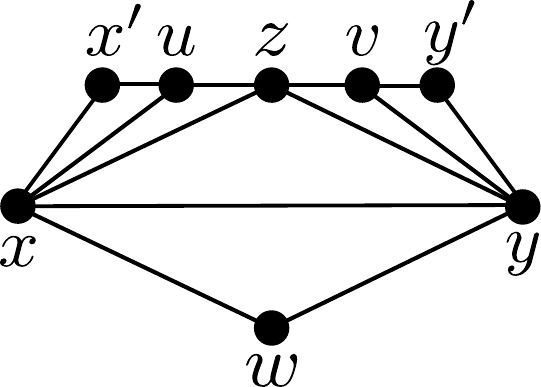}
      \caption{$H_{xy}$ in Case 2.}
      \label{fig:case2}
    \end{figure}

\noindent This completes the proof of Theorem \ref{thm:main}.
    \end{proof}

\section*{Acknowledgments}
Most of the research presented in this paper was conducted at the 2023 MRC
Summer Conference: Ricci Curvatures of Graphs and Applications to Data Science,
supported by the National Science Foundation under Grant Number 1916439. We
sincerely thank the organizers. Anna Schenfisch was partially supported by the Dutch Research Council
(NWO) under project no. P21-13. 
Zhiyu Wang was partially supported by the LA Board of Regents grant LEQSF(2024-27)-RD-A-16.
Jing Yu was partially supported by the NSF
CAREER grant DMS-2239187 (PI: Anton Bernshteyn) during her residency at School
of Mathematics, Georgia Institute of Technology.

We would like to sincerely thank the anonymous referees for their valuable comments and suggestions that greatly improved the manuscript.

\end{document}